\documentclass{article}
\usepackage[utf8]{inputenc}
\usepackage[T1]{fontenc}
\usepackage[english]{babel}
\usepackage{amsthm}
\usepackage{amssymb}
\usepackage{amsbsy}
\usepackage{amsmath}
\usepackage{pgf,tikz}
\usepackage{asymptote}
\usepackage{enumerate}
\usepackage{mathrsfs}
\usetikzlibrary{arrows}
\usepackage{circledsteps}
\usepackage[all]{xy}
\usepackage{ upgreek }
\usepackage{xfrac}
\usepackage[square,comma,numbers,sort&compress,nonamebreak]{natbib}
\usepackage{xcolor}

\newtheorem{theorem}{Theorem}[section]
\newtheorem{lemma}[theorem]{Lemma}

\theoremstyle{definition}
\newtheorem{remark}[theorem]{Remark}
\newtheorem{definition}[theorem]{Definition}

\usepackage{url}

\newcommand\numberthis{\addtocounter{equation}{1}\tag{\theequation}}
\title{Combinatorial spectra using polynomials}
\author{Sylwia Cichacz$^{1}$, Martin Dz\'urik$^2$\\ ~~ \\
\normalsize $^1$AGH University of Science and Technology, al. A. Mickiewicza 30, 30-059 Krakow, Poland\\
cichacz@agh.edu.pl\\
\normalsize $^2$
Masaryk University, Kotl\'a\v rsk\'a 2, CZ-61137 Brno, Czech Republic\\
451859@mail.muni.cz}

\date{\today}

\begin{document}

\maketitle

\begin{abstract}
In
this paper we would like to introduce some new methods for studying magic type-colorings of graphs or domination of graphs, based on combinatorial
spectrum on polynomial rings. We hope that this concept will be potentially
useful for the graph theorists.
    
\end{abstract}
\section{Introduction}
A \textit{Hamiltonian cycle} of a graph  is a cycle  containing all vertices of the graph. A graph (digraph) is \textit{Hamiltonian} if it has a Hamiltonian
cycle. The study for the existence of Hamiltonian cycles is an important topic in graph
theory.  A \textit{Hamiltonian walk} of a graph is a closed walk passing all
vertices of the graph. While not every graph is Hamiltonian, every connected graph
contains a Hamiltonian walk. The Hamiltonian number $h(G)$ of a graph $G$ is the minimum length of a Hamiltonian walk in the graph. The concept of Hamiltonian number was introduced
by Goodman and Hedetniemi for the study of hamiltonicity (see \cite{GH}).

The hamiltonian number can be also defined by cyclic orders. Namely, a Hamiltonian walk can be expressed as a cyclic ordering of $V (G)$. Let
$\pi= (v_0, v_1,\ldots , v_{n-1}, v_n = v_0)$ be a cyclic ordering of $V (G)$. Denote $d(\pi)$ as
the sum of the distances between pairs of consecutive vertices in $\pi$. That is,
$$d(\pi) = \sum_{i=1}^nd(v_i, v_{i+1}).$$
Let $H(G)$ denote the set of values $d(\pi)$ over all cyclic orderings $\pi$ of $V (G)$.
Then the minimum value in $H(G)$ is indeed the Hamiltonian number  $h(G)$ \cite{ChThZh}.

In \cite{Dzurik} the concept of the Hamiltonian spectrum of a graph $G$ was generalized to $H$-Hamiltonian spectrum of the graph $G$ denoted by ${\mathscr{H}}_{H}(G)$.
 Let $G$, $H$ be graphs such that $|V (G)| = |V (H)|$ and
$f \colon V (H) \to V (G)$ be a bijection, then we call $f$ a \textit{pseudoordering on the graph $G$ by $H$}, denote 
$$s_H(f,G) = \sum_{\{x,y\}\in E(H)}d_G(f(x), f(y)).$$
 We will call $s_H(f, G)$ the sum of the pseudoordering $f$. Then
\[{\mathscr{H}}_{H}(G)=\{s_H(f, G) \colon f  \text{ pseudoordering on  $G$ {by} $H$}\}\]
is the $H$-\textit{Hamiltonian spectrum of the graph} $G$.
Observe that the $C_{|V (G)|}$-Hamiltonian spectrum of a graph $G$ is equal to the
Hamiltonian spectrum of $G$ for $|V (G)| \geq 3$.

The approach of $H$-Hamiltonian spectrum led to one more generalization, which is called the \textit{combinatorial spectrum} introduced in \cite{Dz}. This spectrum is now for $R$-weighted graphs,
where $R$ is a ring with operations $+,\cdot$. Let $H$ and $G$ be $R$-weighted complete graphs such that
$|V(G)| = |V(H)|$ with weighted functions
$\nu_H \colon E(H) \to R$, and $\nu_G \colon E(G) \to R$,
and $f \colon V(H) \to V(G)$ be a bijection. We define $H *_f G$ as follows
$$V(H *_f G) = V(H),\;\;\; E(H *_f G) = E(H),\;\;\; \nu_{H*_f G}(e) = \nu_H(e) \cdot \nu_G(f(e)).$$
Let now 
$$H * G = \{H *_f G \colon f \colon V(G) \to V(H)\,\,\mathrm{is}\;\;\mathrm{a}\;\; \mathrm{bijection}\}.$$

For families  $\mathscr{H}$ and $\mathscr{G}$
of graphs we define:
\[
\mathscr{H}*\mathscr{G}=
\bigcup_{\substack{G \in \mathscr{G}\\
H \in \mathscr{H}}} 
H*G.
\]

Let $G$ be an (connected) unweighted graph then we define $R$-weighted complete graphs $I(G)$, $d(G)$ as follows,

\[
V(I(G))=V(d(G))=V(G),
\]
\[
\nu_{I(G)}(e)=
\begin{cases} 
1  & \text{for }e \in E(G)\\
0  &\text{for }e\notin E(G)
\end{cases},
\quad \nu_{d(G)}(\{x,y\}) = d_G(x,y).
\]

Let $G$ be a $R$-weighted complete graph with weighted function $\nu_G$ then we define:
\[
s(G)=\sum_{e \in E(G)} \nu_{G}(e).
\]
And for $\mathscr{G} \in \mathscr{P}(R-Graphs_n)$ we define the combinatorial spectra of $G$ by:
\[
s(\mathscr{G})=\{s(G)|G\in\mathscr{G} \}.
\]
{Where $\mathscr{P}(R-Graphs_n)$ is a class of all sets of $R$-weighted complete graphs of order of $n$.
}

{
We will also need the following definition from \cite{Dz}:

Let $H_1,H_2$ be $R-Graphs_n$ such that $V(H_1)=V(H_2)$ then we define an operation $+$ on $ R-Graphs_n$ as follows
\[
V(H_1+H_2)=V(H_1), \quad 
\nu_{H_1+H_2}(e)=\nu_{H_1}(e)+\nu_{H_2}(e).
\]
We also define an operation on $ \mathscr{P}(R-Graphs_n)$ as follows, let $\mathscr{H}_1,\mathscr{H}_2\in \mathscr{P}(R-Graphs_n)$ such that all graphs have the same set of vertices then  
\[
\mathscr{H}_1+\mathscr{H}_2=\left\{ H_1+H_2\middle|H_1\in \mathscr{H}_1,  H_2\in \mathscr{H}_2              
\right\}
\]

}


The Combinatorial Nullstellensatz is an algebraic technique first
introduced by Noga Alon \cite{Alon} in 1999. Being closely related to the
famous Hilbert’s Nullstellensatz, it has extensive applications in combinatorics
and number theory where various results are obtained by
analyzing roots of well-chosen polynomials.
\begin{theorem}[Combinatorial Nullstellensatz, \cite{Alon}]
Let  $p(x_1,\ldots,x_n)$ be a polynomial in $F[x_1,\ldots,x_n]$  of degree $t$, where $F$  is an arbitrary field.

Suppose that the coefficient of the monomial of maximum degree in $\prod_{i=1}^nx^{t_i}_{{i}}$ is nonzero, where $t=\sum_it_i$  and each  $t_i\geq 0$. Then, if  $S_1,\ldots,S_n$ are subsets of $F$  with $|S_i|>t_i$, there is $(s_1,\ldots,s_n)\in S_1\times\ldots\times S_n$ such that $p(s_1,\ldots,s_n)\neq 0$.
\end{theorem}

The theorem turns out to be extremely
useful in providing a necessary and sufficient condition for a certain graph (or
hypergraph) to be colorable \cite{XB2022}. 
In this paper we would like to introduce some new methods based on the combinatorial spectrum on polynomial rings. We hope that this concept will be potentially useful for graph theorists.
\section{Magic-type of labelings}
Generally speaking, magic-type of labeling of a graph $G=(V,E)$ is a mapping from only $V$ or $E$, or their union $V\cup E$, to a set of labels, which most often is a set of integers or elements of a group. Then the weight of a graph element is typically the sum of labels of the adjacent or incident elements of one or both types. When the weight of all elements is required to be equal, then we speak of a magic-type labeling; when the weights should be all different, then we speak of an anti-magic-type labeling. Probably the best-known problem in this area is the {\em anti-magic conjecture} by Hartsfield and Ringel~\cite{HR}, which claims that the edges of every graph except $K_2$ can be labeled bijectively with integers $1, 2, \dots, |E|$ so that the weight of every vertex is unique. This conjecture is still open.

\subsection{Anti-magic labeling}\label{aml}

For a graph $G=(V,E)$  without any isolated vertex, an \textit{antimagic edge labeling} $\ell\colon E\to \{1,2,\ldots,|E|\}$ is a bijection, such that the \textit{weighted degree}  given by $w(u)=\sum_{\{u,v\}\in E}\ell(\{u,v\})$ is injective. A graph is called \textit{antimagic} if it admits an antimagic labeling.

This labeling was introduced by Hartsfield and Ringel in \cite{HR}, they showed that paths, cycles, complete graphs  ($n\geq 3$) are antimagic and conjectured that all connected graphs besides  are antimagic. Another weaker version of the conjecture is that every regular graph is antimagic except $K_2$. Both conjectures remain unsettled so far. However, in 2004 Alon et al. \cite{ref_AKLRY} showed that the the first conjecture is true for dense graphs. 

Below we show how to express a problem of finding if a labeling is anti-magic in the language of combinatorial spectra.

We start with some definitions. Namely, let $j \in \{1,\dots,|V(G)|\}$, then we define an unweighted graph $S(j)$ as a star with center in the vertex $j$ and set of vertices 
\[
V(S(j))= \{1,\dots,|V(G)|\}.
\]
Recall that for a ring $R$ by $R_n[x]$ we denote the polynomial ring over $R$, where the polynomials are of degree at most $n$. Then we define an $\mathbb{R}_{ {|V(G)|}}[x]$-weighted complete graph
\begin{equation}\label{S}
S[x]=\sum_{j \in \{1,\dots,|V(G)|\} } x^{j-1}\cdot I(S(j)).
\end{equation}
For $j,k \in \{1,\dots,|V(G)|\}$ we  denote by $E(j,k)$ an unweighted graph being an edge between vertices $j$ and $k$ and having the set of vertices
\[
V(E(j,k))=\{1,\dots,|V(G)|\}.
\]
Let now 
\[
E[x]=\sum_{
\begin{subarray}{c}j,k \in \{1,\dots,|V(G)|\} \\ j>k
\end{subarray}
} x^{\psi(j,k)}\cdot I(E(j,k))
\]
be an $\mathbb{R}_{ {|V(G)|}}[x]$-weighted complete graph, where 
$\psi(j,k)$ is the sequence number (starting from 0) in the lexicographic order of the couple $(j,k)$   for $j,k\in \{1,\ldots, |V(G)|\}$ and $j>k$. Thus it is easy to see that
 \[
\psi(j,k)= {{j-1} \choose {2}}+k-1.
\]

\begin{lemma}\label{l1}

Let $G$ be an $\mathbb{N}$-weighted graph, that is not complete,
then $G$ is anti-magic if and only if 
\[
|s(G*I(S(1)))|=|V(G)| \text{ and } s(G*I(E(1,2)))={\{0,1,\dots,|E(G)|\}}.
\]
If $G$ is a complete $\mathbb{N}$-weighted graph then $G$ is anti-magic if and only if 
\[
|s(G*I(S(1)))|=|V(G)| \text{ and } s(G*I(E(1,2)))={\{1,\dots,|E(G)|\}}.
\]

We interpret $G$ as a complete $\mathbb{R}$-weighted graph by adding weight zero to all non-edges of $G$.
\end{lemma}

\begin{proof}
\[s(G*I(S(1)))=\left\{\sum_{\substack{e\in E(G)\\ v\in e}}\nu_G(e) \middle| v \in V(G)\right\},
\] 
hence the condition
\[
|s(G*I(S(1)))|=|V(G)|
\]
means that all these sums are different.
\[
s(G*I(E(1,2)))=\left\{ \nu_G(e) \middle| e\in E(G)\right\}
\]
hence the second condition means that all colors used in $G$ are exactly 
\[
\{1,2,\dots,|E(G)|\} \]
{if $G$ is a complete graph or
\[
\{0,1,2,\dots,|E(G)|\}
\]
otherwise.

}
\end{proof}

The above lemma we can apply to get the following result:

\begin{theorem}\label{t1}
Let $\mathscr{G}$ be a set of $\mathbb{N}_0$-weighted { complete graphs}, then $\mathscr{G}$ contains an anti-magic graph if and only if in $s(\mathscr{G}*(S[x]+x^{|V(G)|}E[x]))$ exists a 
polynomial $p$ such that all coefficients $coef_j(p)$ 
{for $j\in \{0,\dots,|V(G)|-1 \}$}, 
are different and the rest of coefficients form a {set $A$, such that $A\supseteq  \{1,2,\dots,|E(G)|\}$}. Where $coef_j(p)$ is the $j$th coefficient of $p$.
\end{theorem}

\begin{proof}
Let there be a polynomial 
\[
p \in s(\mathscr{G}*(S[x]+x^{|V(G)|}E[x]))
\] 
which satisfies these conditions. Then there exist a graph $G \in \mathscr{G}$ and a pseudoordering $g$ such that 
\[
p = s(G*_g(S[x]+x^{|V(G)|}E[x])). 
\] 
We will show now that $G$ is an antimagic graph. Note that
\[
coef_j(p) = \begin{cases}
s(G*_g I(S(j+1)))  &\text{for } j\in \{0,\dots, n-1\}\\
s(G*_g I(E(\psi^{-1}(j-n)))  &\text{for } j\in \{n,\dots,{n \choose 2}+n-1\}\\
0 &\text{otherwise},
\end{cases}
\]
where $n=|V(G)|$ and $\psi^{-1}$ is the inverse of the map, $\psi $ defined before.
{Note that there exist pseudoorderings $g' : V(G) \rightarrow V(I(S(1)))$ and $g'': V(G) \rightarrow V(I(E(1,2))) $ such that:
 \[g'(g^{-1}(j+1))=1, \;\;\; g''(g^{-1}(\psi^{-1}(j-n)))=\{1,2\}.\]
 Therefore \[
s(G*_g I(S(j+1))) = s(G*_{g'} I(S(1)),
\]
and 
\[
s(G*_g I(E(\psi^{-1}(j-n))) = s(G*_{g''} I(E(1,2))).
\]}

Now we can say that by the first condition, we get that 
\[
|s(G*I(S(1)))|=|V(G)|
\]
and from the second 
\[
s(G*I(E(1,2))){\supseteq}\{1,2,\dots,|E(G)|\}.
\]
Hence $G$ is antimagic by  Lemma \ref{l1}.
For another implication, we can define 
\[
p = s(G*_g(S[x]+x^{|V(G)|}E[x])) 
\] 
and analogously check the conditions.
\end{proof}
{
 Let $\mathscr{G} \in \mathscr{P}(R-Graphs_n)$. If there exists $n \in \mathbb{N}$ such that for all $n\leq k\in \mathbb{N}$
\[
\mathscr{G}^{*k}=\mathscr{G}^{*n},
\]
then as in \cite{Dz} we will denote 
\[
\mathscr{G}^{*\infty}=\mathscr{G}^{*n}.
\]
This is a standard notation from the theory of semigroups.
For $G$ being in $R-Graphs_n$ let $
G^{*\infty}=\{G\}^{*\infty}$.

}
We will need now the following lemma.
\begin{lemma}[\cite{Dz}]\label{ldz}
\[I(G)*\left(\sum_{i=1}^{k-1} \left(I(K_{|V(G)|}\setminus e)^{*\infty} \cup \left\{I(K_{|V(G)|})\right\}\right)+ I(K_{|V(G)|}) \right)\]
is equal to a set of all $k$-colorings (possibly non-proper) of the graph $G$ where we add zeros to non-edges of $G$. 
\end{lemma}

Let us denote 
\[
C(1,\dots,k)_n=\sum_{i=1}^{k-1} \left(I(K_{n}\setminus e)^{*\infty} \cup \left\{I(K_n)\right\}\right)+ I(K_n) .
\]
\begin{remark}\label{irr}
In this notation is the formula from the last lemma equal to 
\[
I(G)*C(1,\dots,k)_{|V(G)|}.
\]
\end{remark}

{

\begin{theorem}
Let $G$ be an unweighted graph then 
$G$ is antimagic if and only if there exists a polynomial
\[
p\in s\left(I(G)*C(1,\dots,|E(G)|)_{|V(G)|}*\left(S[x]+x^{|V(G)|}E[x]\right)\right)
\] 
such that all coefficients $coef_j(p)$ 
for $j\in \{0,\dots,|V(G)|-1 \}$, 
are different and the rest of the coefficients form a
{set $A$, such that $A\supseteq  \{1,2,\dots,|E(G)|\}$}.
\end{theorem}

\begin{proof}
By Theorem \ref{t1} the existence of this polynomial is equivalent to the fact that the set \[
I(G)*C(1,\dots,|E(G)|)_{|V(G)|}
\]
contains an anti-magic graph. By Remark \ref{irr} and Lemma \ref{ldz} this set is equal to the set of all $|E(G)|$-colorings of the graph $G$. 
Hence it is equivalent to the fact that there exists an anti-magic $|E(G)|$-coloring of the graph $G$ and because any anti-magic coloring is $|E(G)|$-coloring we will get the statement of the theorem.
\end{proof}

}

\subsection{Irregular labeling}
It is easy to check, by the Pigeonhole Principle, that in any simple graph $G$ there exist two vertices of the same degree. The situation changes if we consider an edge labeling $\ell\colon E(G)\rightarrow \{1, \ldots, k\}$ (note that the labels do not have to be distinct) and as before calculate {weighted degree} of every vertex $v$ as the sum of labels of all the edges incident to $v$. The labeling $\ell$ is called \textit{irregular} if the weighted degrees of all the vertices are unique (so it is an anti-magic type labeling). The smallest value of $k$ that allows some irregular labeling is called the \textit{irregularity strength of $G$} and is denoted by $s(G)$, but in this article, we will use the notation $\bar{s}(G)$ because it coincides with another our notation. The problem of finding $\bar{s}(G)$ was introduced by Chartrand et al. in \cite{ref_ChaJacLehOelRuiSab1} and investigated by numerous authors  \cite{ref_AigTri2,ref_AmaTog,ref_KalKarPfe1}.

Because irregularity is one part of antimagicness, we can get an equivalence of Lemma \ref{l1} for irregularity.

\begin{lemma}
Let $G$ be a $\mathbb{N}$-weighted graph,
then $G$ is irregular labeled graph if and only if 
\[
|s(G*I(S(1)))|=|V(G)|.
\]
We interpret $G$ as a complete $\mathbb{R}$-weighted graph by adding weight zero to all non-edges of $G$.
\end{lemma}
\begin{proof}
The proof follows from the proof of Lemma \ref{l1}.
\end{proof}

\begin{lemma} \label{irl}
Let $\mathscr{G}$ be a set of $\mathbb{N}_0$-weighted complete graphs, then $\mathscr{G}$ contains an irregular labeled graph if and only if in $s(\mathscr{G}*S[x])$ exists a polynomial $p$ such that all coefficients $coef_j(p)$ are different. Where $S[x]$ is defined in \eqref{S}.
\end{lemma}

\begin{proof}
The proof follows from the proof of Theorem \ref{t1}.
\end{proof}

\begin{theorem}
Let $G$ be an unweighted graph and $k\in \mathbb{N}$ then 
$k<\bar{s}(G)$ if and only if there exist no polynomial
\[
p\in s\left(I(G)*C(1,\dots,k)_{|V(G)|}*S[x]\right)
\]
such that all coefficients $coef_j(p)$ are different, where
$C(1,\dots,k)$ is the combinatorial spectra defined in last subsection.
\end{theorem}

\begin{proof}
Let $k<\bar{s}(G)$, by the definition of $\bar{s}(G)$ there 
exists no irregular $k$-ordering of $G$ now by Lemmas \ref{irl} and \ref{irr} we will get that there exists no polynomial $p$ in given combinatorial spectra such that all coefficients $coef_j(p)$ are different.

The opposite direction is analogous.
\end{proof}

\subsection{1-2-3 conjecture}
The concept of coloring the vertices with the sums of edge labels was introduced by Karo\'nski, {\L}uczak and Thomason (\cite{ref_KarLucTho}). The authors posed the following question. Given a graph $G$ with no component of order less than $3$, what is the minimum $k$ such that there exists a vertex-coloring $k$-edge-labeling? We will say that this coloring is \textit{local irregular labeling}. As there are some analogies with ordinary proper graph coloring,   we will call this minimum value of $k$ the \textit{sum chromatic number} and denote it by $\chi^\Sigma(G)$. 

The well-known 1-2-3 Conjecture of Karo\'nski, {\L}uczak and Thomason asserts
that $\chi^\Sigma(G)\leq 3$ for every graph $G$ with no component of order less than $3$. The first constant bound was proved by Addario-Berry et al. in \cite{ref_AddDalMcDReeTho} ($\chi^\Sigma(G)\leq 30$) and then improved by Addario-Berry et al. in \cite{ref_AddDalRee} ($\chi^\Sigma(G)\leq 16$), Wan and Yu in \cite{ref_WanYu} ($\chi^\Sigma(G)\leq 13$) Kalkowski, Karo\'nski and Pfender in \cite{ref_KalKarPfe2} ($\chi^\Sigma(G)\leq 5$),  and finally by Keusch in \cite{Keusch} ($\chi^\Sigma(G)\leq4$).  

\begin{definition}\label{defM[x]}
Let $j< k \in \{1,\dots,|V(G)|\}$, then we define a $\mathbb{C}$-weighted complete graph $J(j,k)$ on the set of vertices $V(J(j,k))=\{1,\dots,|V(G)|\}$ with weight function
\begin{align*}
\nu_{J(j,k)}(v,u)=\begin{cases}  \phantom{-}i  &\text{if } |\{j,k\} \cap \{v,u\}|=2 \\
\phantom{-}1 &\text{if } \{j,k\} \cap \{v,u\}=\{j\} \\
-1 &\text{if } \{j,k\} \cap \{v,u\}=\{k\} \\
\phantom{-}0 &\text{otherwise}.
\end{cases}
\end{align*}
Then we define a $\mathbb{C}[x]_{\leq {{|V(G)|}\choose{2}}-1}$-weighted complete graph
\[
M[x]=\sum_{
\begin{subarray}{c}j,k \in \{0,\dots,|V(G)|-1\} \\ j<k
\end{subarray}
} x^{\psi(j,k)}\cdot J(j,k),
\]
where $\psi(j,k) $ is the sequence number in lexicographic order {defined in the last subsection.}
\end{definition}

\begin{lemma}
Let $G$ be an $\mathbb{N}$-weighted graph,
then $G$ has a {local irregular labeling} if and only if $s(G*J(1,2))$ does not contain a {nonzero} purely imaginary number.
We interpret $G$  as a complete $\mathbb{C}$-weighted graph by adding weight zero to all non-edges of $G$.\label{lem:ir1}
\end{lemma}

\begin{proof}
Let $s(G*J(1,2))$ contain a nonzero purely imaginary number, by the definition of $J(1,2)$ it needs to be $i$. Let $f$ be a pseudoordering satisfying
\[
s(G*_f J(1,2))=i,
\]
let us rewrite the $s(G*_f J(1,2))$ as follows and denote $f^{-1}(1)=v$ and $f^{-1}(2)=u$.
\[
s(G*_f J(1,2))=
\sum_{w\in N_G(v)\setminus \{u\}} \nu_G(\{v,w\})
-\sum_{w\in N_G(u)\setminus \{v\}} \nu_G(\{u,w\})
+i \cdot \nu_G(\{v,u\})
\]
where $N_G(v)$ is set of neighbors of $v$ in the graph $G$.
This means that there is an edge between vertices $v$ and $u$ and 
\[
\sum_{w\in N_G(v)\setminus \{u\}} \nu_G(\{v,w\})
=\sum_{w\in N_G(u)\setminus \{v\}} \nu_G(\{u,w\}),
\]
hence 
\[
\sum_{w\in N_G(v)} \nu_G(\{v,w\})
=\sum_{w\in N_G(u)} \nu_G(\{u,w\}).
\]
Finally, this says that $G$ is not locally irregular.

Let now $G$ be not locally irregular. This means that there exist vertices $v,u\in V(G)$ such that $\{v,u\}\in E(G)$ and 
\[
\sum_{w\in N_G(v)} \nu_G(\{v,w\})
=\sum_{w\in N_G(u)} \nu_G(\{u,w\}).
\]
If we now define a pseudoordering  $f:V(G)\rightarrow J(1,2)$ by $f(v)=1,f(u)=2$, for the rest of vertices arbitrary, we will get 
\[
s(G*_f J(1,2))=i,
\]
similarly as in the opposite direction.
\end{proof}

\begin{lemma}\label{123l2}
Let $\mathscr{G}$ be a set of $\mathbb{N}_0$-weighted graphs, then $\mathscr{G}$ contains a local irregular labeled graph if and only if there exists in $s(\mathscr{G}*M[x])$ a polynomial $p$ such that all coefficients {of $p$} are not {nonzero} purely imaginary numbers.
\end{lemma}

\begin{proof}
The proof is similar to the proof of Theorem \ref{t1}. Let $G\in \mathscr{G}$ be a local irregular labeled graph. {Note that $
s(G*_f J(1,2)),
$ is not a nonzero purely imaginary number for any pseudoordering $f$  by Lemma~\ref{lem:ir1}.}
Observe that for any $j<k$ we have $J(j,k) \cong J(1,2)$. Hence for any $f$ there exists $f'$ such that
\[ 
s(G*_f J(j,k))=s(G*_{f'} J(1,2)).
\]
This {implies} that $
s(G*_f J(j,k)),
$ is not a nonzero purely imaginary number.
Let $f$ be a pseudoordering $f:V(G)\rightarrow V(M[x])$ then \[
s(G*_f M[x])\in s(\mathscr{G}*M[x])
\]
and 
\begin{align*}\label{eq1}
s(G*_f M[x])&=s(G*_f\sum_{
\begin{subarray}{c}j,k \in \{0,\dots,|V(G)|-1\} \\ j<k
\end{subarray}
} x^{\psi(j,k)}\cdot J(j,k)) \\
&=\sum_{
\begin{subarray}{c}j,k \in \{0,\dots,|V(G)|-1\} \\ j<k
\end{subarray}
} x^{\psi(j,k)}\cdot s(G*_f J(j,k)).
\numberthis
\end{align*} 
As we showed before all of those coefficients 
are not nonzero purely imaginary numbers.

Let now $p \in s(\mathscr{G}*M[x])$ be a polynomial with no nonzero purely imaginary coefficients, this means that 
\[
p=s(G*_f M[x])
\]
for some $G\in \mathscr{G}$ and pseudoordering $f$. Similarly as in the another direction from equation (\ref{eq1}) we get that for all $j<k\in \{0,\dots, |V(G)|-1\}$ $s(G*_f J(j,k))$ is not a nonzero purely imaginary number. Let now $g$ be arbitrarily pseudoordering $g\colon V(G)\rightarrow V(J(1,2))$, we will show that $s(G*_g J(1,2))$ is not a  nonzero purely imaginary number. There exist $j,k$ such that 
\[
s(G*_g J(1,2))=s(G*_f J(j,k)),
\]
indeed, it could be satisfied by 
\[
j=f(g^{-1}(1)),~k=f(g^{-1}(2)).
\]
Hence $s(G*_g J(1,2))$ is  not a nonzero purely imaginary number.

\end{proof}

\begin{theorem}
1-2-3 conjecture holds if and only if for any simple finite unweighted graph $G$
with no connected component isomorphic to $K_2$ exists
\[
p \in s\left(I(G)*C(1,2,3)_{|V(G)|}*M[x]\right)
\]
such that all coefficients of $p$ are not purely imaginary numbers.
\end{theorem}

\proof{
For 1-2-3 conjecture we want that for any $G$ of given properties exists a local irregular labeling of $G$ by colors $\{1,2,3\}$. By Lemma \ref{ldz}
\[
I(G)*C(1,2,3)_{|V(G)|}
\]
is equal to the set of all colorings of $G$ by colors $\{1,2,3\}$. Now by Lemma \ref{123l2} this condition is satisfied if and only if in this set exists a local irregular labeled graph.

}

\section{Domination}

The study of domination in graphs originated around 1850 with the problems of placing a minimum number of queens or other chess pieces on an $n \times n$ chess board to cover/dominate every square.  Formally, for a graph $G = (V,E)$ be a graph. A set $S \subset V$ is a \textit{dominating set} of $G$ if every vertex in $V-S$ is adjacent to some vertex in $S$. The domination number $\gamma(G)$ of $G$ is the minimum cardinality of a dominating set.
Roman domination is a variation of domination suggested by ReVelle  \cite{ref_Re}. This concept and its variation have been widely explored since then.  Emperor Constantine had the requirement for deploying two types of armies, stationary and traveling, in cities such that each city is either protected by a stationary army deployed there or is protected by a traveling army in an adjacent city that also has a stationary army. We may formulate the problem in terms of graphs. Graphs are simple in this paper. A Roman dominating function of a graph $G=(V,E)$ is a function $f\colon V \to\{0,1,2\}$ such that every vertex $v$ with $f(v)=0$ adjacent to some vertex $u$ with $f(u)=2$. The weight of a Roman dominating function $f$ is the value $w(f)=\sum_{v\in V(G)}f(v)$. The \textit{Roman domination number} of $G$, denoted by $\gamma_R(G)$, is the minimum weight of a Roman dominating function of $G$.

Recently, Roushini Leely Pushpam and Malini Mai \cite{ref_RoMa} initiated the study of the edge version of Roman domination. 
An edge Roman dominating function of a graph $G$ is a function $f\colon E\to\{0,1,2\}$ satisfying the condition that every edge $e$ with $f(e)=0$ is adjacent to some edge $e'$ with $f(e')=2$. The \textit{edge Roman domination number} of $G$, denoted by $\gamma'_R(G)$, is the minimum weight $w(f)=\sum_{e\in E(G)}f(e)$ of an edge Roman dominating function $f$ of $G$.

\subsection{Domination number}
\begin{definition}
Let $ k \in \{1,\dots,|V(G)|-1\}$, then we define a $\mathbb{R}[x]$-weighted complete graph $D_k[x]$ on the set of vertices $V(D_k[x])=\{1,\dots,|V(G)|\}$ with the following weight function, let $j<l \in \{1,\dots,|V(G)|\}$
\begin{align*}
\nu_{D_k[x]}(j,l)=\begin{cases}  x^{j-1}  &\text{if } j \leq |V(G)|-k <l\\
0 &\text{otherwise}.
\end{cases}
\end{align*}
\end{definition}

\begin{lemma}
Let $G$ be an unweighted graph, $k\in \{1,\dots,|V(G)|-1\}$ and \[f: V(D_k[x]) \rightarrow V(I(G))\] pseudoordering,
then $f(\{|V(G)|-k+1,\dots , |V(G)|\})$ is a dominating set of $G$ if and only if 
\[
s(D_k[x]*_f I(G))
\]
has all coefficients nonzero.
\end{lemma}

\begin{proof}
Let $D=f(\{|V(G)|-k+1,\dots , |V(G)|\})$ be a dominating set and $j\in \{0,1,\dots |V(G)|-k-1\}$ we will show 
that \[
coef_j\left(s(D_k[x]*_f I(G)) \right) > 0.
\]

If we now consider a vertex $v=f(j+1)$, because $v\notin D$

there exists $u \in D$ such that $\{v,u\}\in E(G)$.
Let us denote $l=f^{-1}(u)$ now 
\[
\nu_{D_k[x]*_f I(G)}(\{j+1,l\})=\nu_{D_k[x]}(\{j+1,l\}) \cdot \nu_{I(G)}(v,u)=x^j \cdot 1.
\]
Hence 
\[
coef_j\left(s(D_k[x]*_f I(G)) \right) > 0.
\]

Let now 
\[
s(D_k[x]*_f I(G))
\]
has all coefficients nonzero.
We need to show that for any \[v \notin f(\{|V(G)|-k+1,\dots , |V(G)|\})\]
there exists an edge $\{v,u\}\in E(G)$ such that $u\in f(\{|V(G)|-k+1,\dots , |V(G)|\})$.
Let us denote $j\in \{1,\dots , |V(G)|-k\}$ such that  $f(j)=v$. Because coefficient \[
coef_{j-1}\left(s(D_k[x]*_f I(G)) \right) > 0,
\]
there exists $l\in \{|V(G)|-k+1,\dots , |V(G)|\}$ such that 
\[
\nu_{D_k[x]*_f I(G)}(\{j,l\})=x^{j-1}.
\]
Hence $\{f(j),f(l)\}\in E(G)$.
\end{proof}

\begin{remark}
If we denote $D=f(\{|V(G)|-k+1,\dots , |V(G)|\}),$ coefficients from the last lemma are equal to the number of neighbors in complement of $D$. 
\[coef_j\left(s(D_k[x]*_f I(G)) \right)=|N_G(f(j+1))\cap (V(G)\setminus D)|
\]
where $N_G(v)$ is a neighborhood of $v$ in $G$.
\end{remark}

\begin{theorem}
Let $G$ be a unweighted graph and $k\in \{1,\dots,|V(G)|-1\}$
then there exists a dominating set of cardinality $k$ if and only if there exists a polynomial 
\[
p\in s(D_k[x]* I(G))
\]
which has all coefficients nonzero.
\end{theorem}

\begin{proof}
Follows from the previous lemma.
\end{proof}

\subsection{Edge Roman domination number}

We start with the following definition:
\begin{definition}
\[
C(0,-1,y)=\bigl\{H\in \mathbb{Z}[y]-\text{Graph}_n\bigm|s(H*E(1,2))\subseteq \{0,-1,y\}\bigr\}
\]
\end{definition}

\begin{lemma}\label{l0-1y}
Let $G$ be an unweighted graph then the following set 
\[
I(G)*C(0,-1,y)
\]
is equal to a set of all colorings of $G$ by colors $\{0,-1,y\}$ (and all no-edges of $G$ are colored also by $0$).

\end{lemma}

\begin{proof}
Every edge weighted by $1$ in $I(G)$ is multiplied by $0,-1,y$ in any possible way.
\end{proof}
For a coloring $f\colon V(G)\to\{-1,0,y\}$ we define $\tilde{f} \colon V(G)\to\{0,1,2\}$ by$$\tilde{f}(v)=\begin{cases}
  0&\text{if } f(v)=-1\\
  1&\text{if } f(v)=0\\
  2&\text{if } f(v)=y.
\end{cases}$$ 

Before we proceed, recall that $s(H) \in \mathbb{Z}[y]$ is a polynomial of degree $1$.
\begin{lemma}\label{w(f)}
Let $G$ be an unweighted graph and $H\in I(G)*C(0,-1,y)$ then \[
w(\widetilde{\nu_H})=|E(G)|+s(H)(1).
\]
\end{lemma}

\begin{proof}
Let us denote $p=s(H)$. Then $-p(0)$ is equal to the number of $-1$ in $H$, hence number of $0$ in $\widetilde{\nu_H}$. Moreover $p(1)-p(0)$ is equal to the number of $y$ in $H$, hence the number of $2$ in $\widetilde{\nu_H}$. Thus \[
|E(G)|-(p(1)-p(0))-(-p(0))=|E(G)|-p(1)+2p(0)
\]
is equal to number of $1$ in $\widetilde{\nu_H}$. And now 
\[
w(\widetilde{\nu_H})=|E(G)|-p(1)+2p(0)+2p(1)-2p(0)=|E(G)|+p(1).
\]
\end{proof}
The next definition will be similar to Definition \ref{defM[x]}.
\begin{definition}
Let $j< k \in \{1,\dots,|V(G)|\}$, then we define a $\mathbb{C}$-weighted complete graph $R(j,k)$ on the set of vertices $V(R(j,k))=\{1,\dots,|V(G)|\}$ with weight function
\begin{align*}
\nu_{R(j,k)}(v,u)=\begin{cases}  i  &\text{if } |\{j,k\} \cap \{v,u\}|=2 \\
1 &\text{if } |\{j,k\} \cap \{v,u\}|=1 \\
0 &\text{otherwise}.
\end{cases}
\end{align*}
Then we define a $\mathbb{C}[x]_{\leq {{|V(G)|}\choose{2}}-1}$-weighted complete graph
\[
R[x]=\sum_{
\begin{subarray}{c}j,k \in \{0,\dots,|V(G)|-1\} \\ j<k
\end{subarray}
} x^{\psi(j,k)}\cdot R(j,k),
\]
where $\psi(j,k) $ is the sequence number in lexicographic order defined in the Subsection \ref{aml}.
\end{definition}

\begin{lemma}
Let $G$ be an unweighted graph and $H\in I(G)*C(0,-1,y)$
then $\widetilde{\nu_H}$ is an edge Roman domination function if and only if 
\[
s(H*R(1,2))
\]
does not contain an unreal constant polynomial, equivalently \[
\forall p \in s(H*R(1,2)) : p \notin-i+\mathbb{Z}.\]
\end{lemma}

\begin{proof}
Let $H\in I(G)*C(0,-1,y)$ such that any edge in $H$ weighted by $-1$ has at least one adjacent edge weighted by $y$ ($\widetilde{\nu_H}$ is Roman dominating function).
Let $g$ be an arbitrary pseudoordering.
If an edge weighted by $0$ is mapped to $i$ by $g$ then 
\[s(H*_g R(1,2)) \in \mathbb{Z}[y].\]
If $y$ is mapped to $i$ then 
\[s(H*_g R(1,2)) \in iy+\mathbb{Z}[y].\]
And if $-1$ is mapped to $i$ then 
\[s(H*_g R(1,2)) \notin -i+\mathbb{Z}\]
because $-1$ has at least one neighbor weighted by $y$.

Another direction is similar.
\end{proof}
The proof of the below lemma is similar to the proof of Lemma~\ref{123l2}.
\begin{lemma}
Let $\mathscr{H}\subseteq I(G)*C(0,-1,y)$, then $\mathscr{H}$ contains a graph satisfying edge Roman condition if and only if there exists in $s(\mathscr{H}*R[x])$ a polynomial $p$ such that all its coefficients $coef_{x^j}(p)\in \mathbb{C}[y]$ are not in $-i+ \mathbb{Z}$.
\end{lemma}

\begin{theorem}
Let $G$ be an unweighted graph and $k<|E(G)|$ then 
$\gamma'_R(G)\leq k$ if and only if there exists a polynomial
\[
p\in\left\{q \in s\bigl(I(G)*C(0,-1,y)*R[x]\bigr)\middle| |E(G)|+\frac{q(1,1)}{2|V(G)|-4+i}\leq k\right\}
\] 
such that all coefficients $coef_{x^j}(p)\notin -i+ \mathbb{Z}$ 
for $j\in \left\{0,\dots,{|V(G)|\choose 2}-1 \right\}$.

\end{theorem}

\begin{proof}
By the previous lemma, the condition for coefficients is equivalent to the fact that in the set \[
I(G)*C(0,-1,y)
\] there is a graph $H$ satisfying the edge Roman dominance condition and by Lemma \ref{l0-1y} this is the right set. It remains to show that the condition \[
|E(G)|+\frac{q(1,1)}{2|V(G)|-4+i}\leq k
\]
is equivalent to that $w(\widetilde{\nu_H})\leq k$.

If we now substitute $x=1$, we will get 
\[
R[1]=\sum_{
\begin{subarray}{c}j,k \in \{0,\dots,|V(G)|-1\} \\ j<k
\end{subarray}
} R(j,k)=
(2|V(G)|-4+i)\cdot I(K_{|V(G)|}).
\]
The last equation holds because the
edge $\{m,l\}$ has weight $1$ in $
2(|V(G)|-2)$
summands, $i$ in one summand, and in the remaining summands it is $0$.
Let now \[q \in s\bigl(I(G)*C(0,-1,y)*R[x]\bigr) \subseteq \mathbb{C}[x,y],\]
after substitution $x=1$, hence
\begin{align*}
q(1,y) &\in s\bigl(I(G)*C(0,-1,y)*(2|V(G)|-4+i)\cdot I(K_{|V(G)|})\bigr)\\
\intertext{and since $s\left(H*I\left(K_{|V(G)|}\right)\right)=s(H)$ we obtain that} 
\frac{q(1,y)}{2|V(G)|-4+i}  &\in s\bigl(I(G)*C(0,-1,y)\bigr).
\end{align*}
Now by Lemma \ref{w(f)} for $
H\in I(G)*C(0,-1,y) $ we have $w(\widetilde{\nu_H})=|E(G)|+s(H)(1)$ and thus
\[
|E(G)|+\frac{q(1,1)}{2|V(G)|-4+i}=w(\widetilde{\nu_H}).
\]
\end{proof}

{
In this paper, we were studying magic type-colorings of graphs or domination of graphs but one can also consider some other graph problems like other colorings of graphs or clique number, independence number and so on. 
}

\bibliography{biblio}

\end{document}